\documentclass[11pt]{amsart}
\usepackage{calc,amssymb,amsmath,ulem,stmaryrd}
\usepackage{alltt}
\RequirePackage[dvipsnames,usenames]{color}

\input{kmacros3.sty}
\input{xy}
\xyoption{all}
\usepackage{hyperref}
\usepackage{graphicx}
\usepackage[all,cmtip]{xy}
\usepackage{verbatim}
\usepackage{diagrams}

\usepackage{mabliautoref}

\theoremstyle{theorem}

\renewcommand{\O}{\mathcal O}

\renewcommand{\to}[1][]{\xrightarrow{\ #1\ }}

\begin{document}
\numberwithin{equation}{theorem}
\title[Uniform bounds for strongly $F$-regular surfaces]{Uniform bounds for strongly $F$-regular surfaces}
\author{Paolo Cascini}
\address{Department of Mathematics, Imperial College London, London SW7 2AZ, UK}
\email{{p.cascini@imperial.ac.uk}}

\author{Yoshinori Gongyo}
\address{Graduate School of Mathematical Sciences, the University of Tokyo, 3-8-1 Komaba, Meguro-ku, Tokyo 153-8914, Japan.}
\email{gongyo@ms.u-tokyo.ac.jp}
\address{Department of Mathematics, Imperial College London, 180 Queen's Gate, London SW7 2AZ, UK.}
 \email{y.gongyo@imperial.ac.uk}

\author{Karl Schwede}
\address{Department of Mathematics\\ The University of Utah\\  155 S 1400 E\\ Salt Lake City\\ UT 84112}
\email{schwede@math.utah.edu}

\begin{abstract}
We show that  if $(X,B)$ is  a two dimensional Kawamata log terminal pair defined over an algebraically closed field of characteristic $p$, and $p$ is sufficiently large, depending only on the coefficients of $B$, then $(X,B)$ is also strongly $F$-regular.
\end{abstract}
\thanks{The first author was partially supported by EPSRC grant P28327}
\thanks{The second author was partially supported by the Grand-in-Aid for Research Activity Start-Up $\sharp$24840009 from JSPS and Research expense from the JRF fund.  }
\thanks{The third author was partially supported by the NSF grant DMS \#1064485, NSF FRG grant DMS \#1265261, NSF CAREER grant DMS \#1252860 and a Sloan Fellowship.}

\maketitle

\section{Introduction}
It has been well understood for a long time that in the study of birational geometry over the complex field, it is
convenient to work with log pairs with mild singularities,  such as Kawamata log terminal singularities. On the other hand, in the study of birational geometry in positive characteristic,
it is not clear what the right category of singularities is. On one side, Kawamata log terminal singularities are the right singularities because they are preserved by the minimal model program, on the other hand, strongly $F$-regular singularities, which are defined via the Frobenius morphism,  are important because they allow to extend many results known over the complex field to  algebraically closed field of positive characteristic.

The aim of this paper is to study the relationship within these two categories of singularities in dimension two. In a sequence of papers, Hara-Watanabe  \cite{HaraWatanabeFRegFPure}, Hara-Yoshida \cite{HaraYoshidaGeneralizationOfTightClosure},
and Takagi \cite{TakagiInterpretationOfMultiplierIdeals} proved that if $(X, B)$ is a log pair defined over the complex numbers, then $(X, B)$ is Kawamata log terminal if and only if its reduction modulo $p$ is strongly $F$-regular, for any sufficiently large $p$. In addition, any strongly $F$-regular pair is always Kawamata log terminal. On the other hand, even in the case of surfaces, it is possible to show the existence of Kawamata log terminal pairs which are not strongly $F$-regular (e.g. see \autoref{ex_sfr} for a sequence of examples in any characteristic).

Here we prove that if $(X,B)$ is  a two dimensional Kawamata log terminal pair defined over an algebraically closed field of characteristic $p$, and $p$ is sufficiently large, depending on the coefficients of $B$, then $(X,B)$ is also strongly $F$-regular.

More specifically, our main theorem is the following:

\begin{theorem}\label{thm.global to local}Let $I\subseteq (0,1)\cap \mathbb Q$ be a finite subset and let $\Gamma=D(I)$ (cf. \autoref{d_hyperstandard}).

Then there exists a positive constant $p_0$ depending only on $I$ such that if $(X,B:=\sum_{i=1}^\ell q_i D_i)$ is a two dimensional Kawamata log terminal pair defined over an algebraically closed field of characteristic $p>p_0$ and such that the  coefficients of $B$ belong to $\Gamma$, then $(X,B)$ is strongly $F$-regular.
\end{theorem}

Note that in the absolute case, i.e. assuming that $B=0$, then the result follows from  \cite{HaraDimensionTwo}, with $p_0=5$.  For an arbitrary finite set $I$, it is possible to find $p_0$ effectively, using the bound of \autoref{cor. Effective A2 finite}, although this bound is not sharp.

\medskip

Thanks to the work of Hacon and Xu \cite{HaconXuThreeDimensionalMMP}, we are able to show the existence of dlt flips, if the characteristic of the underlying field is sufficiently large:

\begin{theorem}\label{t_flips}
Let $I\subseteq (0,1]$ be a finite set. Then there exists a prime $p_I$, depending only on $I$, such that if  $(X,B)$ is a three dimensional  dlt pair over an algebraically closed field of characteristic $p>p_I$,  such that the coefficients of $B$ belong to $I$, and $f\colon X\to Y$ is a $(X,B)$-flipping contraction, then the flips exists.
\end{theorem}

Recently Birkar \cite{BirkarMMP} has proven a stronger version of the result above, by showing that dlt flips always exists over any algebraically closed field of characteristic $p_I>5$, by using different methods than ours.

\medskip

The paper is organised as follows: In  \autoref{s_preliminary}, we introduce the tools used in the rest of the paper. In \autoref{s_uniform}, we prove  \autoref{thm.global to local} in the special case $X=\mathbb A^2$ and where $B$ is a line arrangement. By a standard cone construction, this allows us to show that if  $(\mathbb P^1,B)$ is a log Fano pair defined over a field of characteristic $p$ and $p$ is sufficiently large, depending on the coefficients of $B$, then $(\mathbb P^1,B)$ is globally $F$-regular. Thus, by using a global to local method, this leads to a proof of \autoref{thm.global to local} in the general case, as in \autoref{s_global}. Finally, in \autoref{s_flips}, we prove \autoref{t_flips}.
\medskip

\noindent \textbf{Acknowledgement.} We would like to thank H. Tanaka and  Y. Prokhorov for many useful discussions.
 We would also like to thank the referee for carefully reading the paper and for many helful suggestions.

\section{Preliminary results}\label{s_preliminary}

 We work over an algebraically closed field $k$ of positive characteristic $p$, unless otherwise stated.

We refer to \cite{KollarMori} for the classical definitions of singularities appearing in the minimal model program (e.g. Kawamata log terminal pairs), except for the fact that in our definitions we require the pairs to have effective boundaries.
Given a log pair $(X,\Delta)$ and a geometric valuation $E$ over $X$, we denote by $a(E,X,\Delta)$ the {\em log discrepancy} of $(X,\Delta)$ with respect to $E$.
We say that a pair $(X,\Delta)$ is {\em log Fano} if it is Kawamata log terminal and $-(K_X+\Delta)$ is ample.

We refer to \cite{SchwedeSmithLogFanoVsGloballyFRegular} for the classical definition of singularities in positive characteristic (see also \autoref{d_fpure} and \autoref{d_global}).

Given a subset $I\subseteq [0,1]$, we will say that $I$ is {\em ACC} (respectively {\em DCC}) if it satisfies the ascending chain condition
(respectively the descending chain condition).

\begin{definition}[Hyperstandard set]\label{d_hyperstandard}
Let $I\subseteq [0,1]$ be a subset. We define:
\[
I_+=\{ \sum_{j=1}^m a_j i_j\mid i_j\in I, a_j\in \mathbb N \text{ for } j=1,\dots,m\}\cap [0,1],
\]
and
\[
D(I)=\{ \frac{m-1+f}{m}\mid m\in \mathbb N, f\in I_+\}\cap [0,1].
\]
\end{definition}

The following results are well known:
\begin{lemma}\cite[Lemma 4.4]{McKernanProkhorovThreefoldThresholds}\label{l_di}
Let $I\subseteq [0,1]$ be a subset. Then
\[
D(D(I))=D(I) \cup \{1 \}
\]
and $I$ is DCC if and only if $D(I)$ is DCC.
\end{lemma}

\begin{lemma}\cite[Lemma 4.3]{McKernanProkhorovThreefoldThresholds}\label{l_adjunction}
Let $(X,\Delta)$ be a log canonical pair such that the components of $\Delta$ belong to a subset $I\subseteq [0,1]$, and let $S$ be an irreducible component of $\lfloor \Delta \rfloor$.
Let $\Theta$ be the divisor on $S$ defined by adjunction:
\[
(K_X+\Delta)|_S=K_S+\Theta.
\]
Then,  the coefficients of $\Theta$ belong to $D(I)$.
\end{lemma}

%

\subsection{Some remarks on $F$-pure thresholds}\label{s_fpure}

Suppose that $D \geq 0$ is a divisor on a normal integral \emph{affine} scheme $X = \Spec R$.  We begin by recalling the definition of sharp $F$-purity and the $F$-pure threshold

\begin{definition}[Sharp $F$-purity and strong $F$-regularity]\label{d_fpure}
For any real number $\lambda \geq 0$, we say that $(X, \lambda D)$ is \emph{sharply $F$-pure} if for some $e > 0$ there exists
\[
\phi \in \Hom_{\O_X}\big(F^e_* (\O_X( \lceil (p^e - 1) \lambda D \rceil)), \O_X)
\]
such that $\phi(F^e_* \O_X) = \O_X$.

We say that $(X, \lambda D)$ is \emph{strongly $F$-regular} if for every effective Weil divisor $E \geq 0$ there exists $e > 0$ and
\[
\phi \in \Hom_{\O_X}\big(F^e_* (\O_X( \lceil (p^e - 1) \lambda D + E \rceil)), \O_X)
\]
such that $\phi(F^e_* \O_X) = \O_X$.
\end{definition}

\begin{remark}
In the above definition of sharp $F$-purity, if a single $e > 0$ yields a $\phi$, then all multiples $ne$ of that $e > 0$ also yield elements
\[
\phi^n \in \Hom_{\O_X}\big(F^{ne}_* (\O_X( \lceil (p^{ne} - 1) \lambda D \rceil)), \O_X)
\]
 such that $\phi^n(F^{ne}_* \O_X) = \O_X$.
\end{remark}

\begin{definition}[$F$-pure threshold]
With $(X, D)$ as above, the \emph{$F$-pure threshold of $(X, D)$}, denoted $\fpt(X, D)$, is defined to be
\[
\sup \big\{ t > 0\; | \; (X, tD) \text{ is sharply $F$-pure} \big\}.
\]
In the case that $X$ is $\bQ$-Gorenstein and $D$ is $\bQ$-Cartier, it can be shown that $\fpt(X, D)$ is a rational number \cite{SchwedeTuckerZhang}.
\end{definition}



\subsection{Globally $F$-regular pairs}\label{ss_global}
We now recall the definition of globally $F$-regular pairs:
\begin{definition}[Global $F$-regularity] \cite{HaconXuThreeDimensionalMMP,SchwedeSmithLogFanoVsGloballyFRegular,SmithGloballyFRegular} \label{d_global}
Let $f\colon X\to Y$ be a proper morphism of normal varieties and let $\Delta\ge 0$ be a $\mathbb Q$-divisor on $X$. Then $(X,\Delta)$ is {\em globally $F$-regular over} $Y$ if for any effective divisor $D$, there exists a positive integer $e$ such that
\[
\mathcal O_X \to \mathcal  O_X(\lceil (p^e-1)\Delta\rceil +D)
\]
splits locally over $Y$.
\end{definition}
\begin{remark}
Using the same notation as \autoref{d_global}, if $Y$ is affine then $(X,\Delta)$ is globally $F$-regular over $Y$ if and only if for any effective divisor $D$, there exists a positive integer $e$ such that the  natural map
\[
H^0(X,\mathcal O_X(\lfloor(1-p^e)(K_X+\Delta)\rfloor -D)\to H^0(X,\mathcal O_X)
\]
is surjective \cite[Proposition 2.10]{HaconXuThreeDimensionalMMP}.
\end{remark}

Let $X$ be a normal projective variety and let $L$ be an ample divisor on $X$. We denote by $R(X,L)
=\bigoplus_{m\in \mathbb Z}H^0(X,\mathcal O_X(mL))$  the section ring of $L$. The corresponding affine cone over $X$  is given by
$$W=\Spec R(X,L).$$
For each effective $\mathbb Q$-divisor $\Delta$ on $X$, we denote by $\Delta_W$ the associated $\mathbb Q$-divisor on $W$ (e.g. see \cite[\S 3.1]{Kollar13} for more details).

\begin{proposition}
\label{prop.cone}
Let $X$ be a normal variety, $\Delta\ge 0$ a $\mathbb Q$-divisor on $X$ and $L$  an ample divisor. Let $W$ be the affine cone over $X$ associated to $L$ and let $\Delta_W$ be the corresponding $\mathbb Q$-divisor.

Then $K_W + \Delta_W$ is $\bQ$-Cartier if and only if $r(K_X + \Delta) \sim_{\mathbb Q} L$ for some rational number $r \in \bQ$.  In this case:

\begin{enumerate}
\item[(1)] $(X,\Delta)$ is log Fano if and only if $(W,\Delta_W)$ is Kawamata log terminal.
\item[(2)] $(X,\Delta)$ is globally $F$-regular if and only if $(W,\Delta_W)$ is strongly $F$-regular (in fact, this holds even without the $\bQ$-Cartier assumption).
\end{enumerate}
\end{proposition}
\begin{proof}
(1) is well known (e.g. see \cite[Lemma 3.1]{Kollar13}).
(2) follows from \cite[Proposition 5.3]{SchwedeSmithLogFanoVsGloballyFRegular}
\end{proof}

\begin{lemma}\label{l_sfr}
Let $(X,\Delta)$ be a sharply  $F$-pure pair and let $C$ be an effective divisor such that $(X,\Delta)$ is strongly $F$-regular outside the support of $C$ and $(X,\Delta+\varepsilon C)$ is sharply $F$-pure for some $\varepsilon >0$. Then $(X,\Delta)$ is strongly $F$-regular.
\end{lemma}
\begin{proof}
We may assume that $X$ is affine. Thus, the claim follows from \cite[Corollary 3.10]{SchwedeSmithLogFanoVsGloballyFRegular}.
\end{proof}

The following result is a slight generalization of \cite[Proposition 4.3]{HaraDimensionTwo} and \cite[Propsition 3.8]{HaconXuThreeDimensionalMMP}.
\begin{proposition}
\label{prop.GlobInvOfAdj}
Let $f:X \to Y$ be a birational morphism of normal varieties. Let $\Delta$ and $B$ be $\mathbb{Q}$-divisors on $X$ such that  $(X, \Delta)$ is purely log terminal,  $S = \lfloor \Delta \rfloor$ is prime and normal,  $(X,B)$ is Kawamata log terminal and $B+S\le \Delta$. Assume that
\begin{itemize}
 \item[(1)]$-(K_X+\Delta)$ is $f$-ample,
 \item[(2)] $S$ is $f$-exceptional,
 \item[(3)]If $K_S+\Delta_S =(K_X+\Delta)|_{S}$ is defined by adjunction, then $(S,\Delta_S)$ is a globally $F$-regular pair.\end{itemize}
Then $(X, B)$ is  globally $F$-regular  over $Y$.
\end{proposition}

\begin{proof} By standard perturbation techniques (e.g. see \cite[Lemma 2.8 and  2.13]{HaconXuThreeDimensionalMMP}), we may assume that the $\mathbb Q$-Cartier indices of $K_X+\Delta$ and $K_X+B$ are not divisible by $p$.
We may also assume  that $Y$ is affine.  Let $E$ be an effective divisor on  $X$. We may write $E=n_0S+E'$, where $E'$ is an effective divisor which does not contain $S$ in its support  and $n_0$ is a positive integer.
We may assume that $E'$ is Cartier, after possibly replacing $E$ by a larger divisor.
For any sufficiently divisible  positive integer $e$, we consider the following diagram:

\begin{diagram}
H^0(X, \mathcal O_X((1-p^e)(K_X+\Delta)-E')) &  \rTo^\alpha & H^0(X, \mathcal O_X)\\
\dTo^{\gamma}& &\dTo_\beta\\
H^0(S, \mathcal O_S((1-p^e)(K_S+\Delta_S)-E'|_{S})) &\rTo_\delta &  H^0(S, \mathcal O_S).\\
\end{diagram}
The fact that the diagram commutes follows from the fact that the different $\Delta_S$ coincides with the $F$-different \cite{DasOnStronglyFregInversion}.
Since $(S, \Delta_S)$ is globally $F$-regular, the map  $\delta$ is surjective for $ e \gg 0$.
On the other hand $\gamma$ is sujrective since $-(K_X + \Delta)$ is $f$-ample and $e \gg 0$.
Thus $\beta \circ \alpha$ is surjective. By Nakayama's lemma, so is $\alpha$ near $f(S)$.
Since
$\Delta \geq B$,
we have 
\[
H^0(X, \mathcal O_X((1-p^e)(K_X+\Delta)-E')) \subseteq H^0(X, \mathcal O_X((1-p^e)(K_X+B)-E)).
\]
Thus, the surjectivity of
$$H^0(X, \mathcal O_X((1-p^e)(K_X+B)-E)) \to H^0(X, \mathcal O_X)
$$
for $e \gg 0$ follows.
\end{proof}

\begin{proposition}
\cite[Lemma 2.12]{HaconXuThreeDimensionalMMP}
\label{prop.stronglyFregular}
Let $f\colon X\to T$ be a proper birational morphism of normal varieties such that $(X,\Delta)$ is globally $F$-regular over $T$.

Then $(T,\Delta_T=f_*\Delta)$ is strongly $F$-regular.
\end{proposition}

\begin{proposition}
\label{kollar's comp}Let $(X,\Delta)$ be a two dimensional Kawamata log terminal pair
and let $q \in X$ be a closed point.

Then there exists a birational morphism $f:Y \to X$ such that
\begin{itemize}
\item[(1)]$f$ is an isomorphism over $X \setminus \{q\}$,
\item[(2)]$E=f^{-1}(q)$  is irreducible,
\item[(3)]$-(K_Y+\Delta_Y+E)$ is $f$-ample, where $\Delta_Y=f^{-1}_*\Delta$,
\item[(4)] $(Y, \Delta_Y+E )$ is purely log terminal.
 \end{itemize}
\end{proposition}

We  follow closely the proof in \cite[Remark 1]{xu-fund}. The exceptional divisor $E$ is called the {\em Koll\'ar component} of $(X,\Delta)$.

\begin{proof}
Without loss of generality we may assume that $X$ is affine.
We first show that there exists an effective $\mathbb Q$-Cartier divisor $H$ on $X$ such that there exists exactly one exceptional  divisor $E$ over $X$ with $f(E) = q$, such that $a(E,X,\Delta+H)=0$.
Let $L\ge 0$ be a  $\mathbb Q$-divisor on $X$ such that $(X,\Delta+L)$ is log canonical but not Kawamata log terminal at the point $q$. We may assume that $\lfloor \Delta+L\rfloor =0$. Let $g\colon Z\to X$ be a log resolution of $(X,\Delta+L)$. After possibly replacing $X$ by a smaller open subset, we may assume that $g$ is an isomorphism except over $q$. Then for any $t\in [0,1]$, we may write
$$K_Z+ \Delta_Z+ tL_Z + \sum_{i=1}^k a_i(t)E_i = g^*(K_X+\Delta+tL)$$
where $\Delta_Z=g^{-1}_*\Delta$, $L_Z=g^{-1}_*L$, and $E_1,\dots,E_k$ are prime exceptional divisors.  Furthermore, for each $i=1,\dots,k$,  the function $a_i(t)$ is linear, $a_i(t)\le 1$ for any $t\in [0,1]$ and there exists $i\in \{ 1,\dots,k\}$ such that $a_i(1)=1$. Let $b_1,\dots,b_k$ positive integers so that the divisor  $A=-\sum_{i=1}^k b_i E_i$  is $g$-ample.
Fix $\varepsilon >0$ be a (small)  rational number and choose $\delta>0$ to be sufficiently small so that
$$A'=\varepsilon g^*L+\delta A$$
is effective.
Let $c_i(t)=a_i(t)+\delta b_i$, for any $i=1,\dots,k$.
Then
$$K_Z+ \Delta_Z+ tL_Z + A'+ \sum_{i=1}^k c_i(t)E_i = g^*(K_X+\Delta+(t+\varepsilon)L).$$
Then there exists $t_0\in (0,1)$ such that $\max \{ c_i(t_0) \mid i=1,\dots, k\}=1$. We may assume that $c_1(t_0)=1$.
For any $i=2,\dots,k$, let $\eta_i\ge 0$ be  rational numbers such that $\eta_i>0$ if and only if $c_i(t_0)=1$.
We may assume that $\eta_i$ (and $\varepsilon$) are sufficiently small so that the following holds:  if we denote $c_i=c_i(t_0)-\eta_i$,
$I=\{i=2,\dots,k\mid c_i>0\}$, there exists an effective $\mathbb Q$-divisor
$$A''\sim_{\mathbb Q} A' + \sum_{i \geq 2} \eta_i E_i$$
such that if
$$\Theta=t_0L_Z+A''+ \sum_{i\in I} c_iE_i$$
then
$(Z,E_1+ \Delta_Z+ \Theta)$ is purely log terminal.

We have
$$K_Z+E_1+\Delta_Z + \Theta +\sum_{c_i<0} c_i E_i \sim_{\mathbb Q}g^*(K_X+\Delta+(t_0+\varepsilon )L).$$
Thus, if $H=g_*\Theta$, then $H\sim_{\mathbb Q} (t_0+\varepsilon)L$, $(X,\Delta+H)$ is log canonical and
there exists exactly one exceptional divisor $E = E_1$ over $X$ such that $a(E,X,\Delta+H)=0$, as claimed.

Thus, there exists a birational morphism   $f:Y \to X$  from a normal surface $Y$ which extracts $E$ (e.g. see \cite[Proposition 3.1.2]{ProkhorovComplements}). Note that the proof depends on the existence of a minimal model for a two dimensional dlt pair over $X$, which holds also in positive characteristic (e.g. see \cite[Theorem 0.4]{TanakaSurfaceMMP}).
 In particular, $f$ admits  a unique $f$-exceptional divisor $E$ and (1) and (2) follow. If $\Delta_Y$ and $H_Y$ are the strict transform of $\Delta$ and $H$ on $Y$ respectively, then  $(Y, \Delta_Y+H_Y+E)$ is purely log terminal. Thus, (4) follows.  We may write  $f^*H=H_Y+mE$, for some $m>0$.
Since
\[
K_Y+\Delta_Y+H_Y+E=f^*(K_X+\Delta+H)\sim_{\mathbb{Q},f} 0,
\]
it follows that $-(K_Y+\Delta_Y+E)\sim_{\mathbb Q,f}  - mE$.
 Since  $-E$ is $f$-ample, (3) follows.
\end{proof}

\subsection{Complements}
Similarly to \cite{HaconXuThreeDimensionalMMP}, we will use Shokurov's theory of complements. The results in this section hold over any algebraically closed field.

\begin{definition}Let $f\colon X\to Y$ be a proper morphism of normal varieties and let $(X,D=S+B)$ be a log canonical pair where $S=\lfloor D\rfloor$ and $B\ge 0$ is a $\mathbb Q$-divisor whose support does not contain any component of $S$. Let $n$ be a positive integer.
We say that $D'$ is an $n$-{\em complement} of $(X,D)$ over $Y$ if
\begin{enumerate}
\item[(1)] $n(K_X+D')\sim_f 0$,
\item[(2)] $(X,D')$ is log canonical, and
\item[(3)] $nD'\ge nS + \lfloor (n+1) B\rfloor .$
\end{enumerate}
If $D'\ge D$, then the complement $D'$ is called {\em effective}.
If $(X,D)$ admits an $n$-complement over $Y$, then we say that $(X,D)$ is $n$-{\em complementary} over $Y$.
\end{definition}

\begin{lemma}\label{l_complement}
Let $n$ be a positive integer and let  $I\subseteq (0,1]$ be a finite set.
Let $X=\mathbb P^1$ and let $B$ be a $\mathbb Q$-divisor on $X$ such that
\begin{enumerate}
\item[(1)] the  coefficients of $B$ belong to $D(I)$, and
\item[(2)]  $-(K_X+B)$ is nef.
\end{enumerate}

Then, there exists a positive integer $M$, depending only on $I$ and $n$, such that $(X,B)$ admits an $nq$-complement, for some positive integer $q\le M$.
\end{lemma}

\begin{proof}
The result follows  from \cite[Proposition 5.7]{ProkhorovShokurovComplements}. Note that although the result in \cite{ProkhorovShokurovComplements} is stated only over the complex field,  the proof is characteristic free.
\end{proof}

\begin{theorem}\label{t_complement}
Let $I\subseteq [0,1] \cap \bQ$ be a finite set of rational numbers.  Then there exists a positive integer $N$ such that if $f\colon S\to T$ is a proper map of normal surfaces and $(S,B)$ is a log pair such that
\begin{enumerate}
\item[(1)] $(S,B)$ is Kawamata log terminal,
\item[(2)] $-(K_S+B)$ is $f$-nef, and
\item[(3)] the coefficients of $B$ belong to $I$,
\end{enumerate}
then there exists an effective $m$-complement $B^c$ of $(S,B)$ over $T$ such that $(S,B^c)$ is log canonical but not Kawamata log terminal and $m\le N$.
\end{theorem}

\begin{proof}
Assume that $n$ is a positive integer such that $na$ is integral for any $a\in I$. We want to show that there exists a positive integer $M$, depending only on $I$ such that, if $(S,B)$ is a log pair which satisfies (1),(2) and (3), then $(S,B)$ is $nq$-complementary, for some positive integer $q\le M$. Since $nqB$ is integral, any $nq$-complement of $(S,B)$ is automatically effective.

As in the proof of \autoref{kollar's comp}, there exists a nef $\mathbb Q$-divisor $H\ge 0$ such that $(S,B+H)$ admits exactly one exceptional divisor $E$ over $S$ with log discrepancy zero. We may assume that $(S, B + H)$ has no codimension one log canonical centers and $\mathrm{Supp}\,H$ dose not contain any $f$-exceptional divisors.   We denote $B'=B+H$. Then $(S,B')$ is log canonical but not Kawamata log terminal. Let $g\colon X\to S$ be the birational morphism which extracts $E$ and let $h\colon X\to T$ be the induced morphism. We may write
$$K_X+B'_X=g^*(K_S+B')\qquad\text{and}\qquad K_X+B_X=g^*(K_S+B)$$
for some $B'_X\ge 0$ such that $B'_X=E+\{ B'_X \}$, where $E=\lfloor B'_X\rfloor$, and $B_X\le B'_X$.  Note that $B_X$ is not necessarily effective. We denote by $H'$  the strict transform of $H$ in $X$.

We denote by $B''_X = E \vee B_X$ the divisor obtained by replacing  the coefficient corresponding to  $E$ in $B_X$  by 1 and by keeping  all the other coefficients of $B_X$ unchanged. Thus, we have
\[
B_X\le B''_X\le B'_X\quad\text{and}\quad g_*B''_X=B
\]
since $B'_X-B''_X=H'$.
Note that the difference of $B''_X$ and $B_X$ is only the coefficient of $E$. In addition, we have that  $(X,B''_X)$ is plt.

We claim that $-(K_X+B''_X)$ is nef over $T$. Indeed if we have an $h$-exceptional curve $C$ with $-(K_X+B''_X).C<0$, then $C$ is contained in the support of $B'_X-B''_X$ since  $-(K_X+B'_X)$ is nef over $T$.  However, the support of $B'_X-B''_X$ is equal to the support of $H'$ which is not contracted by $h$. This is a contradiction and therefore  $-(K_X+B''_X)$ is nef over $T$.

Let $\Theta\ge 0$ be the $\mathbb Q$-divisor on $E$ defined by
\[
K_{E}+\Theta=(K_{X}+ B''_X)|_{E}.
\]
By \autoref{l_adjunction}, the coefficients of $\Theta$ belong to $D(I)$. Thus, by  \autoref{l_complement}, there exists a positive integer $M$, depending only on $I$, such that
$(E,\Theta)$ admits an $nq$-complement, for some positive integer $q\le M$.
Using the same methods as in \cite[Proposition 6.0.6]{ProkhorovComplements}, it follows that $(X, B''_X)$ is $nq$-complementary over $T$. Note that the proof in \cite{ProkhorovComplements} relies on the Kawamata-Viehweg vanishing theorem for proper birational morphisms of surfaces, which also holds in positive characteristic (e.g. see \cite[Lemma 2.23]{HaconXuThreeDimensionalMMP}).

 Thus, by \cite[Proposition 4.3.1]{ProkhorovComplements}, $(S,B)$ is $nq$-complimentary over $T$. Since $(X,B''_X)$ is not Kawamata log terminal, it follows that the  complement that we obtained is log canonical but not Kawamata log terminal.
\end{proof}

\section{Uniform bounds on the $F$-pure threshold for line arrangements in $\mathbb{A}^2$}\label{s_uniform}

Throughout the section we fix $X = \bA^2_k = \Spec k[x,y]$ with maximal ideal at the origin $\mfrm = \langle x, y \rangle$.

For each $i = 1, \ldots, \ell$, define a distinct divisor (line) $D_i = \Div(x - \lambda_i y)$ going through the origin and choose integers $a_i > 0$.  Consider the divisor $D = \sum a_i D_i$.  We fix $d = \sum a_i$.

Our first goal is to recall known bounds on the $F$-pure threshold of $(X, D)$. First we handle the highly pathological case when one of the $a_i$s is very large.

\begin{lemma}
\label{lem.DegenCaseOfFPT}
Suppose that for some $i$, $2 a_i \geq d$.

Then
\[
\fpt(X, D) = \lct(X,D)={1 \over a_i}.
\]
\end{lemma}
\begin{proof}
Note that this condition can be satisfied for at most two values of $i$.  In the case where $2 a_i, 2 a_j \geq d$, $i \neq j$, then $a_i = a_j = d/2$ and we have a simple normal crossings pair and the statement is obvious.

Thus we can assume that there is a unique $i$ with $2 a_i \geq d$.
Set $D' = \sum_{j \neq i} a_j D_j$ and fix $d' = \sum_{j \neq i} a_j$.
By $F$-adjunction \cite{HaraWatanabeFRegFPure,SchwedeFAdjunction}, it is sufficient to verify that $(D_i, {1 \over a_i} D'|_{D_i})$ is sharply $F$-pure.  But $D_i \cong \bA^1$ and $D'|_{D_i} = d' O$ where $O$ is the origin.  But $d'/a_i \leq 1$ which proves that $(D_i, {1 \over a_i} D'|_{D_i})$ is sharply $F$-pure and completes the proof of the lemma.
\end{proof}

We recall the following result independently obtained by Hara and Monsky \cite{HaraDimensionTwo}. Similar bounds (which are good enough for our purposes) can also be obtained by modifying the method of \cite[Lemmas 3.3, 3.4]{BhattSinghFPTCalabiYau} which also works in higher dimensions.

\begin{theorem}\textnormal{(\cite[Proposition 3.3]
{HaraMonskyFPureThresholdsAndFJumpingExponents}, \cite[Theorem 17]{MonskyMasonsTheorem})}
\label{prop.BoundOnFPTLineArrangements}
Suppose $D = \sum_{i = 1}^{\ell} a_i D_i$ is a line arrangement through the origin in $X = \bA^2$ as above with $d = \sum a_i$. If $2a_i < d$ for all $i$ then
\[
\fpt(X, D) \geq {2 p - \ell + 2 \over dp }.
\]
\end{theorem}

As an easy corollary we obtain the following.

\begin{corollary}
\label{cor.KLTImpliesFRegLineArrangementBasic}
Fix $X = \mathbb{A}^2$.
Suppose that $\Lambda$ is a set of rational numbers in $(0, 1]$ and $t_0$ is positive real constant such that for each integer $d > 2$, if $\lambda < {2 \over d}$ for some $\lambda \in \Lambda$, then ${2 \over d} - \lambda \geq t_0$.
Suppose
\[
p \geq {1 \over t_0}
\]
Then for any line arrangement $(X, D := \sum_{i = 1}^\ell a_i D_i)$ with integer $a_i$ so that for some $\lambda \in \Lambda$
we have $(X, \lambda D)$ is Kawamata log terminal, then $(X, \lambda D)$ is strongly $F$-regular.
\end{corollary}
Before proving this corollary we make two observations.  Firstly, we notice that $\Lambda$ is bounded away from zero by hypothesis (since ${2 \over d}$ converges to zero). Secondly, the condition on $\Lambda$ is satisfied for any ACC subset of $(0,1]$ which is also bounded away from zero.
\begin{proof}
Note that the cases of $\ell = 1, 2$ are uninteresting as then $(X, D)$ is SNC and so $(X, \lambda D)$ is Kawamata log terminal if and only if it is strongly $F$-regular.
We fix $D = \sum_{i = 1}^\ell a_i D_i$ with $d = \sum_{i = 1}^\ell a_i$ and fix $\lambda \in \Lambda$ such that $(X, \lambda D)$ is Kawamata log terminal.
We need to show that $\lambda < \fpt(X, D)$, we know that $\lambda < \lct(X, D)$.  Note that we may assume that $2 a_i < d = \sum_{i = 1}^\ell a_i$ for all $i$ since otherwise $\lct(X, D) = \fpt(X, D) = {1 \over a_i}$ by \autoref{lem.DegenCaseOfFPT}.  Thus we see that $\lct(X, D) = {2 \over d}$ so that ${2 \over d} - \lambda \geq t_0 \geq {1 \over p}$.

Now observe that
\[
\lambda \leq {2 \over d} - {1 \over p} \leq {2 \over d} - {\ell \over dp} = {2 p - \ell \over dp} < {2p - \ell + 2 \over dp} \leq \fpt(X, D) .
\]
where the last inequality comes from \autoref{prop.BoundOnFPTLineArrangements}.  This completes the proof.
\end{proof}


\begin{example}\label{ex_sfr}
We cannot weaken the hypothesis of \autoref{cor.KLTImpliesFRegLineArrangementBasic} to simply that $\Lambda$ is an ACC set.  For example, suppose that $\Lambda = \{ {1 \over n} \}_{n \in \bN}$.  Then for each prime integer $p$, set $D$ to be the set of $p+1$ distinct lines of $\bA^2_{\bF_p}$.  In other words, $D = \Div( x y (x + y) (x+ 2y) \dots (x + (p-1)y) )$. Then it is easy to check that $\fpt(X, D) = {1 \over p} < {2 \over p+1} = \lct(X, D)$.  In particular, for each prime $p$, $(X, {1\over p} D)$ is not strongly $F$-regular even though it is Kawamata log terminal.
\end{example}

\begin{example}[Standard coefficients]
\label{ex.StandardCoefficients}
Suppose that
$$\Lambda = D(\emptyset) = \{ {n - 1 \over n}\; |\; n \in \bN \}$$ is the set of standard coefficients.  Then the only element $\lambda \in \Lambda$ with $\lambda < 2/d$, for some $d > 2$ is $\lambda = {2 - 1 \over 2} = 1/2$, for $d = 3$.  In particular, we see that
\[
1/6 = 2/3 - 1/2 = \min\{ 2/d - \lambda > 0 \;|\; d > 2, \lambda \in \Lambda \}.
\]
It follows from \autoref{cor.KLTImpliesFRegLineArrangementBasic} that if $p > 6 = {1 \over 1/6 }$ and $\lambda \in \Lambda$, then $(X, \lambda D)$ is Kawamata log terminal if and only if it is strongly $F$-regular where $X = \bA^2$ and $D$ is a line arrangement.
\end{example}

\begin{example}[Simple hyperstandard coefficients]
\label{ex.HyperstandardCoefficients}
Fix a positive integer  $n$.
Suppose that
\[{\arraycolsep=1.4pt\def\arraystretch{1.5}
\begin{array}{rl}
\Lambda_n = & D(\big\{ {1 \over n} \big\})\\
 = & D(\big\{ {a \over n} \; |\; a = 0, \ldots n \big\}) \\
 = & \big\{ {nm + a - n \over nm} \;|\;m \in \bN, a = 0, \ldots n\big\}\\
= & \big\{0, {1 \over n}, {2 \over n}, \ldots, {n-1 \over n}, 1, {1 \over 2}, {n + 1 \over 2n}, \ldots, {2n - 1 \over 2n}, {2 \over 3}, {2n + 1 \over 3n}, \ldots, {3n - 1 \over 3n}, {3 \over 4}, \ldots  \big\}\\
\end{array}
}
\]
We want to find the minimum of ${2 \over d} - \lambda$, with $\lambda\in \Lambda_n$ and $d\in \mathbb N$ such that $d>2$ and ${2\over d-\lambda} >0$.
Note that there are only finitely many $\lambda \in \Lambda_n$ which are less than $2/d$ (and all of those are before the ${2 \over 3}$ in the list above), and we only must consider $d = 3, \ldots, 2n-1$.  Also note we may as well consider $n \geq 3$ since $n = 1, 2$ are already covered by \autoref{ex.StandardCoefficients}.

For $d = 3$, we see that
\[
\begin{array}{rcl}
m_3 = \min\big\{ {2 \over 3} - \lambda > 0\;|\; \lambda \in \Lambda_n \big\} = & \min\{ {2 \over 3} - {n + a \over 2n} > 0\;|\; a = 0, \ldots, n \} & \geq {1 \over 6n}.
\end{array}
\]
Likewise for each $2n-1 \geq d > 3$ we see that
\[
\begin{array}{c}
m_d = \min\big\{ {2 \over d} - \lambda > 0\;|\; \lambda \in \Lambda_n \big\} = \min\Big\{ {2 \over d} - {a \over n} > 0\;|\; a = 0, \ldots, n\Big\}  \geq {1 \over dn} \geq {1 \over (2n - 1) n}.
\end{array}
\]
We set $m = \min\{ m_3, m_4, \ldots, m_{2n - 1} \}$.  

Notice that
\[
{2 \over {2n - 1}} - {1 \over n} = {2n - (2n - 1) \over (2n - 1)n } = {1 \over (2n - 1) n}.
\]
It follows easily that $m = \min\big({1 \over (2n - 1)n}, m_3\big)$.  If $n \geq 4$ then $2n - 1 > 6$ in which case $m = {1 \over (2n - 1)n}$.  Finally, we consider $n = 3$ explicitly.  Note that for $m_3$ we are minimizing ${2 \over 3} - {3 + a \over 6}$ and we see that that is minimized when $a = 0$ in which case $m_3 = {1\over 6} \geq {1 \over (2n-1)n} = {1 \over 15}$.
\end{example}

\medskip

Combining this example with \autoref{cor.KLTImpliesFRegLineArrangementBasic} yields the following result.

\begin{corollary}
\label{cor.HyperstandardCoefficientsBound}
Suppose that $X = \bA^2$ and $D$ is a line arrangement through the origin.
If $n \geq 3$, $\Lambda = D(\{{1 \over n}\})$ and $p > 2n^2 - n$ then for any $\lambda \in \Lambda$, we have that $(X, \lambda D)$ is Kawamata log terminal if and only if it is strongly $F$-regular.
\end{corollary}

For our application, we need to handle  the case when the coefficients of $D$ are rational numbers from $\Lambda$.  This can be more complicated as we shall see.

\begin{corollary}\label{cor.uniform for A2}
Fix $X=\mathbb A^2$. Let  $\varepsilon >0$ be a rational number  and let $\Gamma\subseteq (\varepsilon,1)\cap \mathbb Q$ be a subset which satisfies  ACC.

Then there exists a positive constant $p_0$ depending only on $\Gamma$ (and in particular on $\varepsilon$) such that if the line arrangement $(X,B:=\sum_{i=1}^\ell q_i D_i)$ is a Kawamata log terminal pair defined over a field of characteristic $p>p_0$ and such that $q_1,\dots,q_\ell\in\Gamma$, then $(X,B)$ is strongly $F$-regular.
\end{corollary}

\begin{proof}
Suppose not. Then, for any $i\in \mathbb N$ there exists a pair $$(X,B_i=\sum_{j=1}^{\ell_i} q_{i,j}D_{i,j})$$ defined over an algebraically closed field of  characteristic $p_i$, which is Kawamata log terminal but not strongly $F$-regular and such that $\lim p_i=\infty$ and $q_{i,j}\in \Gamma$.

Since $(X,B_{i})$ is Kawamata log terminal and $q_{i,j}>\varepsilon$, it follows that
$$\ell_{i} <\frac{2}{\varepsilon}.$$
Thus, after possibly taking a subsequence, we may assume that $\ell_i=\ell > 2$ is constant.
By  \autoref{cor.KLTImpliesFRegLineArrangementBasic}, we may assume that the set of coefficients
$$\{q_{i,j}\mid i\in \mathbb N,\quad j=1,\dots,\ell\}$$
is not finite and since $\Gamma$ satisfies ACC, after possibly taking a subsequence, we may assume that $q_i:=\sum_{j=1}^\ell q_{i,j}$ is a strictly decreasing sequence.  Note we may also assume that $q_i > 2q_{i,j}$ for all $i,j$.
We may write
$$B_{i}=\frac 1 {c_i} \sum_{j=1}^\ell b_{i,j} D_{i,j}$$
for some positive integers $c_i,b_{i,1},\dots, b_{i,\ell}$. Let $\lambda_i=\frac{1}{ c_i}$ and let $G_i=\sum_{j=1}^\ell b_{i,j} D_{i,j}$ so that $B_i=\lambda_i G_i$.  Let $d_i=\sum_{j=1}^\ell  b_{i,j}=c_i q_i$. Then by assumption, we have that for any $i\in \mathbb N$,
$$\fpt(X,G_i)\le \lambda_i <\lct(X,G_i).$$

Thus, by \autoref{lem.DegenCaseOfFPT} and \autoref{prop.BoundOnFPTLineArrangements}, we have
$$\frac{2p_i-\ell + 2}{d_ip_i}\le \lambda_i <\frac {2} {d_i}.$$
In particular,
$$1-\frac{\ell-2}{2p_i}=\frac{2p_i-\ell +2}{2 p_i}\le {d_i \lambda_i \over 2} = \frac{d_i}{2c_i}=\frac{q_i}2<1.$$
Since $\lim p_i=\infty$ and $q_i$ is a strictly decreasing sequence, we get a contradiction.
\end{proof}

For our purposes however, we need to handle sets of the form $D(I)$ because those sorts of set will occur via adjunction.

\begin{corollary} \label{cor. Effective A2 finite}
 Let $I\subseteq (0,1)\cap \mathbb Q$ be a finite subset and let $\Gamma=D(I)$.  Fix $\varepsilon=\min I\cup \{\frac 1 2\}$, and
$$\mathcal S=\{q=\sum_{i=1}^\ell q_i\mid q_i\in \Gamma \text{ such that }\sum_{i\neq j} q_i>1 \text{ for any $j=1,\dots\ell$}\}\cap (0,2).$$
Define $Q=\max \mathcal S$ and $p_0 = \lfloor \frac{1-\varepsilon}\varepsilon \cdot \frac 1 {1-Q/2}\rfloor$ and observe it depends only on $I$.

If the line arrangement $(X,B:=\sum_{i=1}^\ell q_i D_i)$ is a Kawamata log terminal pair defined over a field of characteristic $p>p_0$ and such that $q_1,\dots,q_\ell\in\Gamma$,
then $(X,B)$ is strongly $F$-regular.
\end{corollary}
Note that although $\mathcal S$ is not finite, it is easy to check that since $I$ is finite, the maximum of $\mathcal S$ is attained and in particular $Q<2$.

\begin{proof}
Assume that $(X,B=\sum_{i=1}^\ell q_iD_i)$ is a Kawamata log terminal pair which is not strongly $F$-regular and such that $q_i\in \Gamma$.

Assume first that there exists $j \in \{1,\dots,\ell\}$ such that $\sum_{i\neq j} q_i\le 1$. Let
$$C=D_j+\sum_{i\neq j} q_i D_i.$$
Then $C\ge B$ and $(D_j,(C-D_j)|_{D_j})=(\mathbb A^1,\sum_{i\neq j}q_i O)$ where $O\in \mathbb A^1$ is the origin. In particular, $(\mathbb A^1,\sum_{i\neq j}q_i O)$ is sharply $F$-pure and by $F$-adjunction it follows that $(X,C)$ is also sharply $F$-pure. By  \autoref{l_sfr}, it follows that $(X,B)$ is strongly $F$-regular, a contradiction.

We now assume that $\sum_{i\neq j} q_i>1$ for any $j \in \{1,\dots,\ell\}$. Let $q=\sum_{j=1}^\ell q_i$.
Since $(X,B)$ is Kawamata log terminal, we have that $q<2$ and  $\ell<\frac 2 \varepsilon$.  In particular, $q\in \mathcal S$ and therefore $q\le Q$.

We may
write
$$B=\frac 1 c\sum_{i=1}^{\ell} b_i D_i$$ 
for some positive integers $c,b_1,\dots,b_\ell$. Let $\lambda =\frac {1}{c}$ and let $G=\sum_{j=1}^\ell b_i D_i$ so that
$B=\lambda G$. Let $d=\sum_{j=1}^\ell b_i = cq$.

Thus, by \autoref{lem.DegenCaseOfFPT} and \autoref{prop.BoundOnFPTLineArrangements}, we have
$$\frac{2p-\ell + 2}{dp}\le \lambda <\frac {2} {d}.$$
In particular,
$$1-\frac{\ell-2}{2p}=\frac{2p-\ell +2}{2 p}\le \frac{d}{2c}=\frac{q}2<1$$
and it follows
$$p\le {1 \over 2} \Big(\frac {\ell -2}{1-q/2}\Big)<\frac {\frac 2 \varepsilon-2}{2-Q}.$$
Thus, the claim follows.
\end{proof}

\begin{remark}Note that while it might be natural to require that $I$ is an arbitrary DCC set (since if $I$ is finite, $D(I)$ is DCC), there are difficulties with this assumption.  For instance if $I$ is a DCC set containing $\{{2 \over 3} - {1 \over n}\}$ then it is not difficult to see that the conclusion of \autoref{cor.uniform for A2} fails to hold for a configuration of three lines (since if $p \equiv 2 \mod 3$, then $\fpt < \lct = {2 \over 3}$).
\end{remark}

The bound of \autoref{cor. Effective A2 finite} is far from sharp as the following example shows.  Thus we hope that the bound can be substantially improved.

\begin{example}
Suppose that $I = \emptyset$.  Then it is not difficult to see that $\varepsilon = {1 \over 2}$ and $Q = \max \mathcal{S} = {1 \over 2} + {2 \over 3} + {4 \over 5} = 59/30$.  Hence we can take $p_0 = 30$.  This is far from optimal, see \cite[Section 3]{HaconXuThreeDimensionalMMP}
\end{example}

\begin{example}
Next suppose that $I = \{ {1 \over 3} \}$.  We see that $\varepsilon = {1 \over 3}$.  We consider three values for $\ell = 3,4, 5$ (notice $\ell < {2 \over 1/3} = 6$).  Note that the list of valid $q_i$ as in \autoref{ex.HyperstandardCoefficients} is
\[
\begin{array}{c}
 \{ 1, {1 \over 2}, {1 \over 3}, {2 \over 3}, {4 \over 6}, {5 \over 6}, {6 \over 9}, {7 \over 9}, {8 \over 9}, {3 \over 4}, {10 \over 12}, {11 \over 12}, {4 \over 5}, {13 \over 15}, {14 \over 15}, {5 \over 6}, {16 \over 18}, {17 \over 18}, {6 \over 7}, {19 \over 21}, {20 \over 21}, {7 \over 8}, {22 \over 24}, {23 \over 24}, {8 \over 9},\ldots \}\\[6pt] 
= \{ 1, {1 \over 2}, {1 \over 3}, {2 \over 3}, {5 \over 6}, {7 \over 9}, {8 \over 9}, {3 \over 4}, {11 \over 12}, {4 \over 5}, {13 \over 15}, {14 \over 15}, {17 \over 18}, {6 \over 7}, {19 \over 21}, {20 \over 21}, {7 \over 8}, {23 \over 24}, {25 \over 27}, {26 \over 27}, \ldots \}
\end{array}
\]

First we consider $\ell = 5$.  Then we see immediately that any $q_i < {2 \over 3}$ (since the minimum value of the other $q_j$ is ${1 \over 3}$).  But the only such elements of $D(I)$ are $\{ {1 \over 3}, {1 \over 2} \}$ and the ${Q}$ associated to any such sum is even smaller than the $I = \emptyset$ case.

Next we consider $\ell = 4$.  Note we cannot have three of the $q_i$ equalling ${1 \over 3}$ since then their sum would be $\leq 1$.  Hence at most two of the $q_i$ are ${1 \over 3}$ and so since the smallest value a third $q_i$ can be is ${1 \over 2}$, we see that any such $q_i < 2 - ({ 1 \over 3} + {1 \over 3} + {1 \over 2}) = {5 \over 6}$.  The only values of $D(I)$ less than ${5 \over 6}$ are $\{ {1 \over 3}, {1 \over 2}, {2 \over 3}, {7 \over 9}, {3 \over 4}, {4 \over 5} \}$.  A quick computation shows that ${Q} = {1 \over 3} + {1 \over 3} + {1 \over 2} + {4 \over 5} = {59 \over 30}$ again and we recover nothing new compared to the $I = \emptyset$ case.

Finally we consider $\ell = 3$.  Suppose now that $q_1 \leq q_2 \leq q_3$.  Of course, at least one of the $q_i < {2 \over 3}$ since their sum is less than $2$ and so either $q_1 = {1 \over 3}$ or $q_1 = {1 \over 2}$.

Suppose first that $q_1 = {1 \over 3}$.  It follows that $q_2 < {5 \over 6}$ and also since $q_1 + q_2 > 1$ we know that $q_2 \in \{ {7 \over 9}, {3 \over 4}, {4 \over 5} \}$.  We break this up into cases.
\begin{description}
\item[$q_1 = {1 \over 3}, q_2 = {7 \over 9}$] and so $q_1 + q_2 = {10 \over 9}$ and $q_3 < {8 \over 9}$.  The largest possible $p_3$ is then ${13 \over 15}$.  We see that $Q \geq {1 \over 3} + {7 \over 9} + {13 \over 15} = {89 \over 45}$.
\item[$q_1 = {1 \over 3}, q_2 = {3 \over 4}$] and so $q_1 + q_2 = {13 \over 12}$ and $q_3 < {11 \over 12}$.  The largest possible $q_3$ is then ${19 \over 21}$.  We see that $Q \geq {1 \over 3} + {3 \over 4} + {19 \over 21} = {167 \over 84}$.
\item[$q_1 = {1 \over 3}, q_2 = {4 \over 5}$] and so $q_1 + q_2 = {17 \over 15}$ and $q_3 < {13 \over 15}$.  The largest possible $q_3$ is then ${6 \over 7}$.  We see that $Q \geq {1 \over 3} + {4 \over 5} + {6 \over 7} = {209 \over 105}$.
\end{description}
Now we handle the case that $q_1 = {1 \over 2}$.  It follows that $q_2 < {3 \over 4}$ and also since $q_1 + q_2 > 1$, we know that $q_2 \in \{ {2 \over 3} \}$.  Hence we see that $q_1 + q_2 = {7 \over 6}$ and so $q_3 < {5 \over 6}$.  The largest possible $q_3$ is then ${4 \over 5}$ and we see that $Q \geq {1 \over 2} + {2 \over 3} + {4 \over 5} = {29 \over 30}$.

Summarizing, we see that $Q = {209 \over 105}$.  Since $\varepsilon = {1 \over 3}$ we see that
\[
p_0 = \Big({ { 1 - (1/3) \over 1/3}\Big) \Big( {1 \over 1 - 209/210}\Big) } = 2 \cdot 210 = 420.
\]
This completes the example.
\end{example}

%



Removing the explicit bound of \autoref{thm.global to local} we can restate \autoref{cor. Effective A2 finite} as follows.

\begin{corollary} \label{cor. A2 finite}
 Let $I\subseteq (0,1)\cap \mathbb Q$ be a finite subset and let $\Gamma=D(I)$.

Then there exists a positive constant $p_0$ depending only on $\Gamma$ such that if the line arrangement $(X:=\mathbb A^2,B:=\sum_{i=1}^\ell q_i D_i)$ is a Kawamata log terminal pair defined over a field of characteristic $p>p_0$ and
$q_1,\dots,q_\ell\in\Gamma$, then $(X,B)$ is strongly $F$-regular.
\end{corollary}

\section{Global to local}\label{s_global}
The aim of this section is to prove \autoref{thm.global to local}. We begin with the following:

 \begin{corollary}
 \label{cor. P1 case}
Let  $I\subseteq (0,1)\cap \mathbb Q$ be a finite subset and let $\Gamma=D(I)$.

Then there exists a positive constant $p_0$ depending only on $I$  such that if $(\mathbb{P}^1,B:=\sum_{i=1}^\ell q_i P_i)$ is a log Fano pair defined over an algebraically closed field of characteristic $p>p_0$ and such that  the coefficients of $B$ belong to $\Gamma$, then $(\mathbb{P}^1,B)$ is globally $F$-regular.\end{corollary}

\begin{proof} By considering the cone associated to $\mathcal L=\mathcal O_{\mathbb P^1}(1)$, \autoref{prop.cone} implies that
$(\mathbb{P}^1,\sum_{i=1}^\ell q_i P_i)$ is globally $F$-regular (resp. log Fano) if and only if $(\mathbb A^2, \sum_{i=1}^\ell q_i D_i)$ is strongly $F$-regular (resp. Kawamata log terminal),
where  $D_1,\dots, D_\ell$ are some suitable distinct lines
through the origin.

 Thus,  \autoref{cor. A2 finite} implies the claim.  In fact we can even choose $p_0$ effectively using the bound of \autoref{cor. Effective A2 finite}.
\end{proof}

We can now proceed with the proof of \autoref{thm.global to local}.

\begin{proof}[Proof of \autoref{thm.global to local}]
Let $B$ be a $\mathbb{Q}$-divisor such that $(X,B)$ is Kawamata log terminal. We may assume that $X$ is affine and let $q\in X$ be a closed point. By  \autoref{kollar's comp}, we can take a birational morphism $f:Y \to X$ such that
\begin{itemize}
\item[(1)]$f$ is isomorphic over $X \setminus \{q\}$,
\item[(2)]$E=f^{-1}(q)$  is irreducible,
\item[(3)]$-(K_Y+B_Y+E)$ is $f$-ample, where $B_Y=f^{-1}_*B$,
\item[(4)] $(Y, B_Y+E )$ is purely log terminal.
 \end{itemize}

In particular, if we write $K_E+B_E=(K_Y+B_Y+E)|_E$ then $(E,B_E)$ is Kawamata log terminal and by  \autoref{l_adjunction}  and \autoref{l_di}, the coefficients of $B_E$ belong to $D(I)$.
Thus,  \autoref{cor. P1 case} implies that there exists a positive integer $p_0$  such that $(E,B_E)$ is globally $F$-regular, if $p>p_0$.
Then   \autoref{prop.GlobInvOfAdj} implies that $(Y, B_Y)$ is globally $F$-regular over $X$.
Thus, by   \autoref{prop.stronglyFregular}, $(X,B)$ is strongly  $F$-regular.
\end{proof}



\section{On the existence of flips}\label{s_flips}

This section is highly inspired by Section 3 in \cite{HaconXuThreeDimensionalMMP}. In particular, our goal is to extend \cite[Theorem 3.1]{HaconXuThreeDimensionalMMP} by the following:

\begin{theorem}\label{t_globallyFregular}
Let $I\subseteq (0,1)$ be a finite set. Then there exists a prime $p_I$, depending only on $I$, such that if  $(S,B)$ is a pair over an algebraically closed field of characteristic $p>p_I$, $f\colon S\to T$ is a morphism of normal surfaces, such that
\begin{enumerate}
\item $(S,B)$ is Kawamata log terminal,
\item $-(K_S+B)$ is $f$-nef, and
\item the coefficients of $B$ belong to $I$,
\end{enumerate}
then $(S,B)$ is globally $F$-regular over $T$.
\end{theorem}

Following the same steps as in \cite[\S 4.4]{HaconXuThreeDimensionalMMP}, it is possible to show that \autoref{t_globallyFregular} implies \autoref{t_flips}.
\medskip

We now begin the proof of \autoref{t_globallyFregular}. Let $B^c\ge B$ and $N$ as in \autoref{t_complement} and let $\nu\colon \tilde S\to S$ be a smooth dlt model of $(S,B^c)$. We may write
$$K_{\tilde S}+B_{\tilde S}^c=\nu^*(K_S+B^c)\qquad\text{and}\qquad K_{\tilde S}+B_{\tilde S} =\nu^*(K_S+B),$$
where $(S,B_{\tilde S}^c)$ is divisorially log terminal and $B_{\tilde S}\le B^c_{\tilde S}$. Note that $B_{\tilde S}$ is not necessarily effective. By \autoref{t_complement}, it follows that $C=\lfloor B_{\tilde S}^c\rfloor$ is not zero.
We first assume that $(\tilde S, B^c_{\tilde S})$ is purely log terminal.

\begin{claim}\label{claim_globallyFregular}
There exist a finite subset $J\subseteq [0,1]$ depending only on $I$ and a $\mathbb Q$-divisor $B^*_{\tilde S}\ge 0$ such that
\begin{enumerate}
\item[(1)] $B_{\tilde S}\le B^*_{\tilde S}\le B^c_{\tilde S}$
\item[(2)] the coefficients of $B^*_{\tilde S}$ belong to $J$,
\item[(3)] $\lfloor B^*_{\tilde S}\rfloor=C$, and
\item[(4)] $-(K_{\tilde S}+ B^*_{\tilde S})$ is nef over $T$ and $-(K_{\tilde S}+ B^*_{\tilde S})|_C$ is ample.
\end{enumerate}
\end{claim}

We now prove how \autoref{claim_globallyFregular} implies  \autoref{t_globallyFregular}. We may write
$$(K_{\tilde S}+ B^*_{\tilde S})|_C=K_C+B^*_C.$$
In particular $(C,B^*_C)$ is log Fano.  By \autoref{cor. P1 case}, there exists $p_0$ depending only on $J$ (and therefore depending only on $I$) such that if the characteristic of the base field is greater than $p_0$ then $(C,B^*_C)$ is globally $F$-regular. Thus, \autoref{t_globallyFregular} follows immediately from \autoref{prop.GlobInvOfAdj}.

\medskip

We now proceed with the proof of \autoref{claim_globallyFregular}. The following Lemma is obvious:
\begin{lemma}\label{easy-lem}
Let $I\subseteq [0,1]$ be a finite set. Then  for any $\varepsilon>0$, the set $D(I)\cap (0,1-\varepsilon)$ is finite.
\end{lemma}
In particular, because our $I$ is finite, we apply Lemma \ref{easy-lem} for $\varepsilon=1/N$. Then there exists a sufficiently small rational number $x > 0$ such that
for any $a\in D(I)$ we have that
$$a\notin \Big(\frac {p-x}{q-x},\frac pq\Big)\qquad\text{for any }q=2,\dots,N,\quad p=1,\dots,q-1.$$
We use $J$ to denote
$$\Big\{\frac {p-x}{q-x}\mid q=2,\dots,N,\quad p=1,\dots,q-1\Big\}\cup \Big\{\frac {p}{q}\mid q=2,\dots,N,\quad p=1,\dots,q-1\Big\}.$$

Let $g=f\circ \nu\colon \tilde S\to T$.  Note that $C$ is $g$-exceptional. In addition, since $(S,B)$ is Kawamata log terminal and $-(K_S+B)$ is $f$-nef, it follows that $(T,B_T)$ is Kawamata log terminal. Thus, since $g$ factors through the minimal resolution of $S$, it follows that the graph associated to the exceptional divisor  of $g$ is a tree and its components are smooth rational curves which meet transversally with each other.  We now proceed similarly as in \cite[Lemma 3.3]{HaconXuThreeDimensionalMMP}. Let $D$ be the connected component of
$B_{\tilde S}^c-B_{\tilde S}$ which contains $C$. Assume by contradiction that $D$ is $g$-exceptional. Then
$$\begin{aligned}
D^2 &=(B_{\tilde S}^c - B_{\tilde S})\cdot D \\
&= (K_{\tilde S} + B_{\tilde S}^c)\cdot D  - (K_{\tilde S}+B_{\tilde S})\cdot D\\
&=-(K_S+B)\cdot \nu_*D\ge 0.
\end{aligned}
$$
which is a contradiction.
Thus, there exists a curve contained in the support of $D$ which is not $g$-exceptional.
Thus, there exists a chain of smooth rational curves in $\tilde S$
$$C_0=C,C_1,\dots,C_k$$ such that, $C_{i-1}\cdot C_{i}=1$ if $i=1,\dots,k-1$,
$C_0,\dots,C_{k-1}$ are $g$-exceptional, $C_k$ is not $g$-exceptional and $C_1,\dots,C_k$ are contained in the support of $B^c_{\tilde S}$. Let $q \in \{2,\dots,N\}$ such that $q(K_{\tilde S}+B^c_{\tilde S}) \sim_{\mathbb{Z},S} 0$.
Then the coefficients of $B^c_{\tilde S}-C$ are of the form $p/q$ for some $p\in \{1,\dots,q-1\}$.
We define $B^*_{\tilde S}$ by replacing, for each $i=1,\dots,k$, the coefficient in $B^c_{\tilde S} $ of the form $p/q$ by the
coefficient $(p-x)/(q-x)$, and for each other curve in the support of $B^c_{\tilde S}$ we do not change the coefficient. In particular,
we have that  $B^*_{\tilde S}\le B^c_{\tilde S}$, $\lfloor B^*_{\tilde S}\rfloor =C$ and the coefficients of  $B^*_{\tilde S}-C$ belong
to $J$.

It follows immediately that $-(K_{\tilde S}+B^*_{\tilde S})\cdot C<0$. Thus, $-(K_{\tilde S}+B'_{\tilde S})|_C$ is ample.

We now show that $-(K_{\tilde S}+B^*_{\tilde S})$ is nef over $T$. To this end, we use the same argument and the same notation as in \cite[Lemma 3.4]{HaconXuThreeDimensionalMMP}.
It is enough to show that $-(K_{\tilde S}+B^*_{\tilde S}) \cdot C_j\ge 0$ for any $j=1,\dots,k-1$.
  For any $j<k-1$, it holds that
$$\frac {p_{j-1} }q +\frac {p_{j+1}}q +\frac{p_j}{q} C^2_j+\frac r q -2 -C^2_j=0$$
 from that
 $$(K_{\tilde S}+B^c_{\tilde S}) \cdot C_j=0.
 $$

This implies that
$$p_{j-1}+p_{j+1}+r-2q +(p_j-q)C^2_j=0$$
which, in turn, implies that
$$\frac{p_{j-1}-x}{q-x} +\frac {p_{j+1}-x}{q-x} +\frac{p_j-x}{q-x} C^2_j+\frac r q -2 -C^2_j\le 0.$$
Thus, $-(K_{\tilde S}+B^c_{\tilde S})\cdot C_i\ge 0$ for any $j=1,\dots,k-2$.
A similar calculation yields  $-(K_{\tilde S}+B^c_{\tilde S})\cdot C_{k-1}< 0$ and the claim follows.

Finally, by the same proof as \cite[Lemma 3.5]{HaconXuThreeDimensionalMMP} and since by assumption the coefficients of $B$ are not contained within the interval $(\frac {p-x}{q-x},\frac p q)$ for any  coefficient $p/q$ of $B^c_{\tilde S}$, it follows that $B_{\tilde S}\le B^*_{\tilde S}$.
If $(\tilde S,B^c_{\tilde S})$ is purely log terminal, then $(\tilde S,B^*_{\tilde S})$ is also purely log terminal and therefore \autoref{claim_globallyFregular} follows.

\medskip
We now assume that $(\tilde S, B^c_{\tilde S})$ is not purely log terminal.
Let
$$J=\{\frac {p}{q}\mid q=2,\dots,N,\quad p=1,\dots,q-1\}.$$

Proceeding exactly as in \cite[pag. 19]{HaconXuThreeDimensionalMMP}, we can find a $\mathbb Q$-divisor $B^*_{\tilde S}\ge 0$ which satisfies the following properties:
\begin{enumerate}
\item[(1)] $(S,B^*_{\tilde S})$ is purely log terminal, and $C=\lfloor B^*_{\tilde S}\rfloor$ is a smooth rational curve,
\item[(2)] $B_{\tilde S} \le B^*_{\tilde S}\le B^c_{\tilde S}$,
\item[(3)] $-(K_{\tilde S}+B^*_{\tilde S})$ is nef over $T$,
\item[(4)] $-(K_{\tilde S}+ B^*_{\tilde S})|_C$ is ample, and
\item[(5)] if we write $(K_{\tilde S}+B^*_{\tilde S})|_C=K_C+\Theta$, there exists a divisor $\Theta'\ge \Theta$ on $C$ whose coefficients belong to $D(J)$ and such that $(C,\Theta')$ is log Fano.
\end{enumerate}
By \autoref{cor. P1 case}, it follows that if $p$ is sufficiently large, depending only on $J$ (and therefore depending only on $I$), then $(C,\Theta')$ is globally $F$-regular.
Thus, $(C,\Theta)$ is also globally $F$-regular, and the statement follows again by  \autoref{prop.GlobInvOfAdj}.

\bibliographystyle{skalpha}
\bibliography{MainBib}

\end{document}